\documentclass[12pt, reqno]{amsart}

\usepackage{amssymb,latexsym}
\usepackage{mathrsfs}
\usepackage{a4wide}
\usepackage{amsthm}
\usepackage[v2,tips]{xy}
\usepackage[T1]{fontenc}
\usepackage[latin1]{inputenc}

\newtheorem{theorem}{Theorem}[section]
\newtheorem{lemma}[theorem]{Lemma}

\newtheorem{corollary}[theorem]{Corollary}

\theoremstyle{definition}
\newtheorem{example}[theorem]{Example}

\numberwithin{equation}{section}

\newcounter{smallromans}

\newenvironment{romanenumerate}
{\begin{list}{{\normalfont\textrm{(\roman{smallromans})}}}%
  {\usecounter{smallromans}\setlength{\itemindent}{0cm}%
   \setlength{\leftmargin}{5.5ex}\setlength{\labelwidth}{5.5ex}%
   \setlength{\topsep}{.5ex}\setlength{\partopsep}{.5ex}%
   \setlength{\itemsep}{0.1ex}}}%
{\end{list}}

\newcommand{\romanref}[1]{{\normalfont\textrm{(\ref{#1})}}}

\newcounter{smallromansdash}

{\end{list}}

\newcounter{bigromans} 
  {\end{list}}

\newcounter{smallarabics}
{\end{list}}

\newcounter{smallalphs}

{\end{list}}

\newcommand{\alphref}[1]{{\normalfont\textrm{(\ref{#1})}}}

\renewcommand{\leq}{\ensuremath{\leqslant}}
\renewcommand{\le}{\ensuremath{\leqslant}}
\renewcommand{\geq}{\ensuremath{\geqslant}}

\newcommand{\N}{\mathbb{N}}
\newcommand{\Z}{\mathbb{Z}}
\newcommand{\R}{\mathbb{R}}
\newcommand{\C}{\mathbb{C}}

\newcommand{\ip}{\operatorname{IP}}
\newcommand{\inv}{\operatorname{inv}}

\newcommand{\spa}{\operatorname{span}}

\newcommand{\tr}{\operatorname{Tr}}

\newcommand{\smashw}[2][l]{{\text{\makebox[0pt][#1]{$#2$}}}}

\renewcommand{\epsilon}{\ensuremath{\varepsilon}}
\renewcommand{\phi}{\ensuremath{\varphi}}

\title[Banach spaces with $K_0(\mathscr{B}(X))\cong\mathbb{Z}$]{Banach
  spaces whose algebra of bounded operators has the integers as their
  $K_0$-group}

\author[Kania]{Tomasz Kania}%
\email{tomasz.marcin.kania@gmail.com}%
\address{Institute of Mathematics, Polish Academy of Sciences,
  ul. \'Sniadeckich 8, 00-956 War\-sza\-wa, Poland.}
\author[Koszmider]{Piotr Koszmider}\email{P.Koszmider@Impan.pl}%
\address{Institute of Mathematics, Polish Academy of Sciences,
  ul. \'Sniadeckich 8, 00-956 War\-sza\-wa, Poland.}
\author[Laustsen]{Niels Jakob Laustsen}\email{n.laustsen@lancaster.ac.uk}%
\address{Department of Mathematics and Statistics, Fylde College
  Lancaster University, Lancaster LA1 4YF, United Kingdom.}

\subjclass[2010]{Primary: 
19K99,
47L10;
Secondary: 19A49,
46E15}

\date{} \keywords{$K_0$-group; Banach algebra; bounded, linear
  operator; Banach space; continuous functions on the first
  uncountable ordinal interval.}
\begin{document}

\begin{abstract} 
  Let $X$ and $Y$ be Banach spaces such that the ideal of operators
  which factor through~$Y$ has codimension one in the Banach
  alge\-bra~$\mathscr{B}(X)$ of all bounded opera\-tors on~$X$, and
  suppose that~$Y$ contains a complemented subspace which is
  isomorphic to~$Y\oplus Y$ and that $X$ is isomorphic to $X\oplus Z$
  for every complemented subspace~$Z$ of~$Y$. Then the $K_0$-group
  of~$\mathscr{B}(X)$ is isomorphic to the additive group~$\mathbb{Z}$
  of integers.

  A number of Banach spaces which satisfy the above conditions are
  identified. No\-ta\-bly, it follows that
  $K_0(\mathscr{B}(C([0,\omega_1])))\cong\mathbb{Z}$,
  where~$C([0,\omega_1])$ denotes the Banach space of scalar-valued,
  continuous functions defined on the compact Hausdorff space of
  ordinals not exceeding the first uncountable ordinal~$\omega_1$,
  endowed with the order topology.
\end{abstract}

\maketitle

\section{Introduction}\label{section1}
\noindent
The purpose of this note is to prove that, for certain Banach
spaces~$X$, the $K_0$-group of the Banach algebra~$\mathscr{B}(X)$ of
(bounded, linear) operators on~$X$ is isomorphic to the additive
group~$\mathbb{Z}$ of integers. More precisely, our main result, which
will be proved in Section~\ref{proofofmainthm}, is as follows.

\begin{theorem}\label{thm26Nov2013}
  Let $X$ and~$Y$ be Banach spaces such that:
  \begin{romanenumerate}
  \item\label{thm26Nov2013ii} $Y$ contains a complemented subspace
    which is isomorphic to~$Y\oplus Y;$
  \item\label{thm26Nov2013iii} $X$ is isomorphic to $X\oplus Z$ for
    every complemented subspace~$Z$ of~$Y;$ and
  \item\label{thm26Nov2013i} the ideal $\{TS : S\in\mathscr{B}(X,Y),\,
    T\in\mathscr{B}(Y,X) \}$ of operators on~$X$ that factor
    through~$Y$ has codimension one in~$\mathscr{B}(X)$.
  \end{romanenumerate}
  Then the mapping
  \begin{equation}\label{thm26Nov2013eq1}
    n\mapsto n[I_X]_0,\quad\mathbb{Z}\to K_0(\mathscr{B}(X)), 
  \end{equation} 
  is an isomorphism of abelian groups, where $[I_X]_0$ denotes the
  $K_0$-class of the identity opera\-tor on~$X$.
\end{theorem}

\noindent
As a consequence, we shall deduce in Section~\ref{section4} that the
$K_0$-group of~$\mathscr{B}(X)$ is isomorphic to~$\Z$ for a number of
Banach spaces~$X$, including the following: 
\begin{romanenumerate}
\item $X = C([0,\omega_1])$, the Banach space of scalar-valued,
  continuous functions defined on the compact Hausdorff space of
  ordinals not exceeding the first uncountable ordinal~$\omega_1$,
  endowed with the order topology;
\item $X = C(K)$ or $X = C(K)\oplus Y$, where~$K$ is the compact
  Hausdorff space constructed by the second-named author
  in~\cite{Koszmidermrowka}, assuming either the Con\-tinu\-um
  Hypothesis or Mar\-tin's Axiom together with the negation of the
  Continuum Hypothesis, and where~$Y$ is a Banach space which is
  isomorphic to the $\ell_p$- or $c_0$-direct sum of countably many
  copies of itself for some $p\in[1,\infty)$, $Y$~contains a
  complemented subspace that is iso\-mor\-phic to~$c_0$, and no
  complemented subspace of~$Y$ is isomorphic to~$C(K)$;
\item $X = X_{\text{AH}}\oplus C_p$, where $X_{\text{AH}}$ is Argyros
  and Haydon's Banach space which solves the scalar-plus-compact
  problem (see~\cite{ah}), $p\in[1,\infty]$, and $C_p$ is Johnson's
  $p^{\text{th}}$ universal space through which all approximable
  operators factor (see~\cite{joh});
\item $X = W\oplus Y$, where~$W$ is the non-separable Banach space
  constructed by Shelah and Stepr\={a}ns~\cite{ss} such that every
  operator on~$W$ is a scalar multiple of the identity plus an
  operator with separable range, and $Y$ is the $\ell_p$-direct sum of
  a certain family of separable subspaces of~$W$ for some
  $p\in(1,\infty)$; see Example~\ref{ShelahSteprans} for details.
\end{romanenumerate}

We also obtain a couple of known results as consequences of
Theorem~\ref{thm26Nov2013}, namely that $K_0(\mathscr{B}(X))\cong\Z$
for $X = J_p$ or $X = J_p(\omega_1)$, where $p\in(1,\infty)$ and $J_p$
denotes the $p^{\text{th}}$ quasi-reflexive James space, while
$J_p(\omega_1)$ denotes Edgar's long version of it (see~\cite{james}
and~\cite{edgar}, respectively).

It is known that the $K_0$-group of~$\mathscr{B}(X)$ vanishes for most
``classical'' Banach spaces~$X$, including every Banach space~$X$
which is primary and isomorphic to its square~$X\oplus X$ (see
\cite[Proposition~2.3]{lauJLMS}). By contrast, the existence of an
ideal of finite codimension in~$\mathscr{B}(X)$ implies that the
$K_0$-class~$[I_X]_0$ of the identity operator on~$X$ is an element of
infinite order in~$K_0(\mathscr{B}(X))$. Theorem~\ref{thm26Nov2013}
can therefore be viewed as a minimality result
for~$K_0(\mathscr{B}(X))$: by condition~\romanref{thm26Nov2013i},
$[I_X]_0$~has infinite order in~$K_0(\mathscr{B}(X))$,
and~\eqref{thm26Nov2013eq1} states that this element generates the
whole group.

\section{Preliminaries}\label{section2}
\noindent
We shall begin by outlining the definition of the $K_0$-group of a
unital ring~$\mathscr{A}$; further details can be found in standard
texts such as~\cite{bl} and~\cite{rll}. For $m,n\in\N$, we denote
by~$M_{m,n}(\mathscr{A})$ the additive group of $(m\times n)$-matrices
over~$\mathscr{A}$. We write $M_n(\mathscr{A})$ instead
of~$M_{n,n}(\mathscr{A})$; this is a unital ring.  Define
\[ \ip_n(\mathscr{A}) = \{ P \in M_n(\mathscr{A}) : P^2 = P\},\quad
\text{the set of idempotent}\ (n\times n)\text{-matrices over}\
\mathscr{A}. \] Given $P\in\ip_m(\mathscr{A})$ and
$Q\in\ip_n(\mathscr{A})$, where $m,n\in\N$, we say that~$P$ and~$Q$
are \emph{algebraically equivalent}, written \mbox{$P\sim_0 Q$}, if $P
= AB$ and $Q = BA$ for some $A \in M_{m,n}(\mathscr{A})$ and $B\in
M_{n,m}(\mathscr{A})$. This defines an equivalence relation~$\sim_0$
on the set $\ip_{\infty}(\mathscr{A}) =
\bigcup_{n\in\N}\ip_n(\mathscr{A})$, and the quotient $V(\mathscr{A})
= \ip_{\infty}(\mathscr{A})/\mathord{\sim_0}$ is an abelian semigroup
with respect to the operation
\begin{equation}\label{K0operation} \bigl([P]_V,[Q]_V\bigr)\mapsto
  \left[\begin{pmatrix}P & 0 \\ 0 & Q \end{pmatrix}\right]_V,\quad
  V(\mathscr{A})\times V(\mathscr{A})\to
  V(\mathscr{A}), \end{equation} where $[P]_V$ denotes the equivalence
class of $P\in\ip_{\infty}(\mathscr{A})$ in~$V(\mathscr{A})$.  The
$K_0$-group of~$\mathscr{A}$, denoted by $K_0(\mathscr{A})$, is now
defined as the Grothen\-dieck group of $V(\mathscr{A})$. The
fundamental property of the Grothen\-dieck group implies that we have the
following \emph{standard picture} of $K_0(\mathscr{A})$:
\begin{equation}\label{sixtheq}
  K_0(\mathscr{A}) = \bigl\{[P]_0 - [Q]_0 :
  P,Q\in\ip_{\infty}(\mathscr{A})\bigr\},
\end{equation}
where $[P]_0$ is the canonical image of $[P]_V$ in $K_0(\mathscr{A})$.

Let $n\in\N$, and suppose that $P,Q\in\ip_n(\mathscr{A})$ are
\emph{orthogonal}, in the sense that $PQ = 0 = QP$. Then $P + Q$ is
idempotent, and the formula for addition in $K_0(\mathscr{A})$ takes
the following simple form:
\begin{equation}\label{fourtheq} [P]_0 + [Q]_0 = [P + Q]_0.
\end{equation}

We shall require one more basic property of~$K_0$: given a ring
homomorphism \mbox{$\phi\colon\mathscr{A}\to\mathscr{C}$}
(where~$\mathscr{C}$, like~$\mathscr{A}$, is a unital ring, but $\phi$
need not be unital) and $n\in\N$, we can define a ring homomorphism
$\phi_n\colon M_n(\mathscr{A})\to M_n(\mathscr{C})$ by entry\-wise
appli\-ca\-tion: \begin{equation}\label{defnphin}
  \phi_n\bigl((A_{j,k})_{j,k=1}^n\bigr) =
  \bigl(\phi(A_{j,k})\bigr)_{j,k=1}^n. \end{equation} This induces a
group homomorphism $K_0(\phi)\colon K_0(\mathscr{A})\to
K_0(\mathscr{C})$ which satisfies
\begin{equation}\label{K0functor} K_0(\phi)([P]_0) = [\phi_n(P)]_0\qquad 
  (n\in\N,\,P\in\ip_n(\mathscr{A})).
\end{equation}

In Sections~\ref{section1}--\ref{section4}, all Banach spaces and
algebras are considered over a fixed scalar field $\mathbb{K} = \R$ or
$\mathbb{K} = \C$, whereas in Appendix~\ref{section5}, we shall
consider complex scalars only.

It is well known and elementary that the
standard (un\-normalized) trace \[ \tr_n\colon\
(\lambda_{j,k})_{j,k=1}^n\mapsto \sum_{j=1}^n \lambda_{j,j},\quad
M_n(\mathbb{K})\to\mathbb{K}, \] induces an isomorphism
$K_0(\tr)\colon K_0(\mathbb{K})\to\Z$ which satisfies
\begin{equation}\label{generator27Feb2013} 
  K_0(\tr)([P]_0)  = \tr_n (P)\qquad (n\in\N,\, P\in\ip_n(\mathbb{K}))
\end{equation} 
(see for instance \cite[Example~3.3.2]{rll} for a proof for
$\mathbb{K} = \C$; the proof for $\mathbb{K} = \R$ is similar).

The symbol $C(K)$ denotes the Banach space of scalar-valued,
continuous functions defined on a compact Hausdorff space~$K$.  By an
\emph{operator}, we understand a bounded, linear mapping between
Banach spaces.  For Banach spaces~$X$ and~$Y$, we
write~$\mathscr{B}(X,Y)$ for the Banach space of all opera\-tors
from~$X$ into~$Y$, and we identify $M_{m,n}(\mathscr{B}(X,Y))$ with
the Banach space $\mathscr{B}(X^n,Y^m)$ of operators from~$X^n$
into~$Y^m$, where $X^n$ denotes the direct sum of $n$ copies of~$X$,
equipped with the norm $\bigl\|(x_1,\ldots,x_n)\bigr\| =
\max\bigl\{\|x_1\|,\ldots,\|x_n\|\bigr\}$.  We write~$\mathscr{B}(X)$
instead of~$\mathscr{B}(X,X)$; this is a unital Banach algebra. We
denote by~$I_X$ the identity operator on~$X$.

The following easy observation clarifies the meaning of the
relation~$\sim_0$ in this case.

\begin{lemma}\label{isomorphicranges}
  Let $X$ be a Banach space, and let
  $P,Q\in\ip_\infty(\mathscr{B}(X))$. Then $P\sim_0 Q$ if and only if
  the ranges of~$P$ and~$Q$ are isomorphic.
\end{lemma}

We shall also require the following related result.

\begin{lemma}[{\cite[Lemma~3.9(ii)]{laustsenMaximal}}]%
\label{STSTlemma}
Let $X$ and $Y$ be Banach spaces, and let $S\in\mathscr{B}(X, Y)$ and
$T\in\mathscr{B}(Y, X)$ be operators such that $ST$ is
idempotent. Then $TSTS$ is idempotent, and the ranges of~$ST$
and~$TSTS$ are isomorphic.
\end{lemma}

Given Banach spaces~$X$, $Y$ and~$Z$, we define
\[ \mathscr{G}_Y(X,Z) = \spa\{TS : S\in\mathscr{B}(X,Y),\,
T\in\mathscr{B}(Y,Z) \}. \] This is an operator ideal in the sense of
Pietsch provided that~$Y$ is non-zero. The `span' is not necessary
when $Y$ contains a complemented subspace which is isomorphic
to~\mbox{$Y\oplus Y$}; in this case the set $\{TS :
S\in\mathscr{B}(X,Y),\, T\in\mathscr{B}(Y,Z) \}$ is automatically a
linear subspace of~$\mathscr{B}(X,Z)$. In line with standard practice,
we write $\mathscr{G}_Y(X)$ instead of~$\mathscr{G}_Y(X,X)$.

\section{The proof of
  Theorem~\ref{thm26Nov2013}}\label{proofofmainthm}
\noindent
Throughout this section we shall suppose that~$X$ and~$Y$ are Banach
spaces which satisfy
con\-di\-tions~\romanref{thm26Nov2013ii}--\romanref{thm26Nov2013i} of
Theorem~\ref{thm26Nov2013}. The third of these conditions implies that
we can de\-fine a bounded, unital algebra homomorphism
$\phi\colon\mathscr{B}(X)\to\mathbb{K}$ by
\begin{equation}\label{eqDefnphi} \phi(\lambda I_X +
  T) = \lambda\qquad (\lambda\in\mathbb{K},\, T\in\mathscr{G}_Y(X)). 
\end{equation}
Recall that~$M_n(\mathscr{B}(X))$ has been identified with
$\mathscr{B}(X^n)$ for each $n\in\N$. Under this identification, we
have $M_n(\mathscr{G}_Y(X)) = \mathscr{G}_Y(X^n)$
because~$\mathscr{G}_Y$ is an operator ideal, and hence
\begin{equation}\label{matrixfactorization}
  \ker\phi_n = M_n(\ker\phi) = \mathscr{G}_Y(X^n) = \{ TS : 
  S\in\mathscr{B}(X^n,Y),\, T\in\mathscr{B}(Y,X^n)\},
\end{equation}
where the final equality follows from the fact that~$Y$ satisfies
condition~\romanref{thm26Nov2013ii}.

The following lemma is the key step in the proof of the surjectivity
of the mapping~\eqref{thm26Nov2013eq1}.

\begin{lemma}\label{claim1}
  Let $P\in \ip_n(\mathscr{B}(X))$ for some $n\in\N$, and
  set $k = \tr_n\circ\,\phi_n(P)$. Then
  \[ [P]_0 = k\cdot [I_X]_0\qquad \text{in}\qquad
  K_0(\mathscr{B}(X)). \]
\end{lemma}

\begin{proof}
  We shall first establish the result for $k=0$. In this case we have
  $\phi_n(P) = 0$ because the zero matrix is the only idempotent,
  scalar-valued matrix with trace zero, so that $P=TS$ for some
  operators $S\colon X^n\to Y$ and $T\colon Y\to X^n$
  by~\eqref{matrixfactorization}. Lemma~\ref{STSTlemma} then implies
  that the operator $Q = SPT\in\mathscr{B}(Y)$ is idempotent with
  $Q[Y]\cong P[X^n]$. Combining this with
  condition~\romanref{thm26Nov2013iii}, we obtain $X\oplus P[X^n]\cong
  X\oplus Q[Y]\cong X$; that is, the operators
  $\bigl(\begin{smallmatrix} I_X & 0\\ 0 & P\end{smallmatrix}\bigr)$
  and~$I_X$ have isomorphic ranges. Hence $[I_X]_0 + [P]_0 = [I_X]_0$
  in $K_0(\mathscr{B}(X))$ by~\eqref{K0operation} and
  Lemma~\ref{isomorphicranges}, so that $[P]_0 = 0$, as required.

  We shall next consider the case where $k=n$. Then
  $\tr_n\circ\,\phi_n(I_{X^n} - P) = 0$, so that $[I_{X^n}-P]_0 = 0$
  by the result established in the first paragraph of the
  proof. Hence, by~\eqref{fourtheq} and~\eqref{K0operation}, we
  conclude that
  \[ [P]_0 = [P]_0 + [I_{X^n}-P]_0 = [I_{X^n}]_0 = n\cdot
  [I_X]_0\qquad \text{in}\qquad K_0(\mathscr{B}(X)). \]

  Finally, suppose that $k\in\{1,2,\ldots,n-1\}$.  Since $\phi_n(P)$
  is an idempotent, scalar-valued matrix, it is diagonalizable, so
  that there exists $R\in \ker\phi_n$ such that $\Delta_k+R$ is
  idempotent and $P\sim_0 \Delta_k+R$, where
  \[ \Delta_k = \begin{pmatrix} I_{X^k} & 0\\ 0 &
    0 \end{pmatrix}\in\ip_n(\mathscr{B}(X)). \]
  By~\eqref{matrixfactorization}, we can find operators $S\colon
  X^n\to Y$ and $T\colon Y\to X^n$ such that $R = TS$. Moreover,
  $\Delta_k$ has an obvious factorization as $\Delta_k = VU$, where
  $U\colon X^n\to X^k$ and $V\colon X^k\to X^n$ denote the projection
  onto the first~$k$ coordinates and the embedding into the first~$k$
  coordinates, respectively.  Condition~\romanref{thm26Nov2013iii}
  implies that there exists an isomorphism $W\colon X^k\to X^k\oplus
  Y$, and we then have a commutative diagram
  \[ \spreaddiagramrows{5ex}\spreaddiagramcolumns{5ex} \xymatrix{%
    X^n\ar[rr]^-{\displaystyle{\Delta_k+R}} \ar[d]_(0.5)
    {\begin{pmatrix} U\\ S\end{pmatrix}} & & X^n\\ X^k\oplus
    Y\ar[r]^-{\displaystyle{W^{-1}}} & X^k\ar[r]^-{\displaystyle{W}} &
    X^k\oplus Y\smashw{,}\ar[u]_(0.5){\begin{pmatrix} V &
        T \end{pmatrix}}} \] where the operators
  $\bigl(\begin{smallmatrix} U\\ S\end{smallmatrix}\bigr)$ and $(V\;\
  T)$ are given by $x\mapsto (Ux, Sx)$ and $(x,y)\mapsto Vx + Ty$,
  respectively. Hence Lemma~\ref{STSTlemma} shows that the operator
  \[ Q = W^{-1}\begin{pmatrix}U\\S\end{pmatrix}(\Delta_k +
  R)\begin{pmatrix}V & T\end{pmatrix}W \in \mathscr{B}(X^k) \] is
  idempotent and the ranges of~$Q$ and~$\Delta_k+R$ are isomorphic, so
  that $Q\sim_0 P$ by Lemma~\ref{isomorphicranges}. The trace property
  implies that $\tr_k\circ\,\phi_k(Q) = \tr_n\circ\,\phi_n(\Delta_k +
  R) = k$, and therefore, as shown in the second paragraph of the
  proof, we have $k\cdot[I_X]_0 = [Q]_0 = [P]_0$, as required.
\end{proof}

\begin{proof}[Proof of Theorem~{\normalfont{\ref{thm26Nov2013}}}]
  We shall show below that the group
  homomorphism \begin{equation}\label{thm26Nov2013eq2} K_0(\tr)\circ
    K_0(\phi)\colon K_0(\mathscr{B}(X))\to\Z \end{equation} is an
  isomorphism.  Since $K_0(\tr)\circ K_0(\phi)([I_X]_0) = 1$, which
  generates the group~$\Z$, the mapping given
  by~\eqref{thm26Nov2013eq1} is the inverse of this isomorphism, and
  hence the conclusion follows.

  The surjectivity of the homomorphism~\eqref{thm26Nov2013eq2} is
  immediate because, as observed above, its range contains the
  generator~$1$ of the group~$\Z$.

  To see that the homomorphism~\eqref{thm26Nov2013eq2} is injective,
  suppose that $g\in\ker K_0(\tr)\circ K_0(\phi)$. By~\eqref{sixtheq},
  we have $g = [P]_0 - [Q]_0$ for some $P\in\ip_m(\mathscr{B}(X))$ and
  $Q\in\ip_n(\mathscr{B}(X))$, where $m,n\in\N$.
  Using~\eqref{K0functor} and~\eqref{generator27Feb2013}, we obtain
  \[ 0 = K_0(\tr)\circ K_0(\phi)(g)= \tr_m\circ\, \phi_m(P) -
  \tr_n\circ\, \phi_n(Q). \] This implies that $[P]_0 = [Q]_0$ by
  Lemma~\ref{claim1}, so that $g=0$. \end{proof}

\section{Applications}\label{section4}
\begin{example}\label{mrowkaexample}
  Assuming either the Continuum Hypothesis or Martin's Axiom together
  with the negation of the Continuum Hypothesis, the second-named
  author~\cite{Koszmidermrowka} has constructed a scattered compact
  Hausdorff space~$K$ such that:
  \begin{enumerate}
  \item\label{mrowkaexample1} the ideal~$\mathscr{X}(C(K))$ of
    operators with separable range has codimension one
    in $\mathscr{B}(C(K));$
  \item\label{mrowkaexample2} every separable subspace of~$C(K)$ is
    contained in a subspace which is isomorphic to~$c_0;$
  \item\label{mrowkaexample3} whenever~$C(K)$ is decomposed into a
    direct sum of two closed, infinite-dimensional subspaces~$A$
    and~$B$, either $A\cong c_0$ and $B\cong C(K)$, or \emph{vice
      versa.}
  \end{enumerate}
  We claim that $X = C(K)$ and $Y = c_0$ satisfy
  conditions~\romanref{thm26Nov2013ii}--\romanref{thm26Nov2013i} of
  Theorem~{\normalfont{\ref{thm26Nov2013}}}. Indeed,
  \romanref{thm26Nov2013ii}~is clear, while~\romanref{thm26Nov2013i}
  follows from condition~\alphref{mrowkaexample1}, above, together
  with~\cite[Theorem~5.5]{KK}, where it is shown that
  $\mathscr{X}(C(K)) = \mathscr{G}_{c_0}(C(K))$. Finally, to
  verify~\romanref{thm26Nov2013iii}, we observe that:
  \begin{itemize}
  \item $C(K)$ contains a complemented subspace which is isomorphic
    to~$c_0$ because $K$ is scattered, and consequently $C(K)\cong
    C(K)\oplus c_0$ by~\alphref{mrowkaexample3}, above; and
  \item every complemented subspace~$Z$ of~$c_0$ is either
    finite-dimensional or isomorphic to~$c_0$, and thus $c_0\oplus
    Z\cong c_0$.
  \end{itemize}
  Hence $C(K)\oplus Z\cong C(K)\oplus c_0\oplus Z\cong C(K)\oplus
  c_0\cong C(K)$, as required.
  Thus we conclude that $K_0(\mathscr{B}(C(K)))\cong\Z$. It is not
  known whether a compact space with the same properties as~$K$ can be
  constructed within ZFC.
\end{example}

In order to facilitate further applications of Theorem~\ref{thm26Nov2013},
we shall show that certain standard prop\-er\-ties of Banach spaces ensure
that the first two conditions of Theorem~\ref{thm26Nov2013} are
satisfied.
\begin{lemma}\label{lemma29Nov2013}
  Let $X$ and~$Y$ be Banach spaces such that~$X$ is isomorphic
  to~$X\oplus Y$, and suppose that~$Y$ satisfies (at least) one of the
  following two conditions:
  \begin{enumerate}
  \item\label{lemma29Nov2013i} $Y$ is isomorphic to the $\ell_p$- or
    the $c_0$-direct sum of countably many copies of it\-self for some
    $p\in[1,\infty);$ or
  \item\label{lemma29Nov2013ii} $Y$ is primary and contains a
    complemented subspace which is isomorphic to~$Y\oplus Y$.
  \end{enumerate}
  Then
  conditions~\romanref{thm26Nov2013ii}--\romanref{thm26Nov2013iii} of
  Theorem~{\normalfont{\ref{thm26Nov2013}}} are satisfied.
\end{lemma}
\begin{proof}
 Condition~\romanref{thm26Nov2013ii} of Theorem~\ref{thm26Nov2013} is
 evidently satisfied in both cases.

  To verify condition~\romanref{thm26Nov2013iii}, suppose that~$Z$ is a
  complemented subspace of~$Y$.

  In case~\alphref{lemma29Nov2013i}, we observe that~$Y\cong Y\oplus
  Y$, so that~$Y$ contains a complemented subspace which is isomorphic
  to~$Y\oplus Z$. On the other hand, $Y\oplus Z$ evidently contains a
  complemented subspace which is isomorphic to~$Y$. Hence the
  Pe\l{}czy\'{n}ski decomposition method (as stated in
  \cite[Theorem~2.23(b)]{ak}, for instance) implies that $Y\cong
  Y\oplus Z$. Combining this with the assumption that $X\cong X\oplus
  Y$, we obtain
  \begin{equation}\label{lemma29Nov2013eq1} 
    X\oplus Z\cong X\oplus Y\oplus Z\cong X\oplus Y\cong X, 
  \end{equation}
  as required.

  In case~\alphref{lemma29Nov2013ii}, either $Y\cong Z$ or $Y\cong
  Y\oplus Z$ because~$Y$ is primary. In the first case, $X\oplus
  Z\cong X$ is immediate from the fact that $X\cong X\oplus Y$, and in
  the second case the calculation~\eqref{lemma29Nov2013eq1}, above,
  applies to give this conclusion.
\end{proof} 

\begin{example}\label{applicationsofmainthmComega1}
  For an ordinal~$\alpha$, denote by~$[0,\alpha]$ the compact
  Hausdorff space consisting of all ordinals not exceeding~$\alpha$,
  endowed with the order topology, and set $X = C([0,\omega_1])$ and
  $Y = \bigl(\bigoplus_{\alpha<\omega_1} C([0,\alpha])\bigr)_{c_0}$,
  where~$\omega_1$ denotes the first uncountable ordinal. Then: 
  \begin{itemize}
  \item $X\cong X\oplus Y$ by \cite[Lemma~2.14(iv) and
    Corollary~2.16]{KKL};
  \item $Y$ is isomorphic to the $c_0$-direct sum of countably many
    copies of itself by \cite[Lemma~2.12]{KKL}, so that
    condition~\alphref{lemma29Nov2013i} of Lemma~\ref{lemma29Nov2013}
    is satisfied (in fact, condition~\alphref{lemma29Nov2013ii} is
    also satisfied by \cite[Corollary~1.3]{KL});
  \item the ideal $\mathscr{G}_Y(X)$ has codimension one
    in~$\mathscr{B}(X)$ by \cite[Theorem~1.6]{KKL}.
  \end{itemize}
  Hence Lemma~\ref{lemma29Nov2013} and Theorem~\ref{thm26Nov2013}
  apply, so that $K_0(\mathscr{B}(X))\cong\Z$.

  This example provided the original motivation behind
  Theorem~\ref{thm26Nov2013}. We have since learnt that a different
  approach is possible for complex scalars, using a result of
  Edelstein and Mityagin, and this result also shows that the
  $K_1$-group of $\mathscr{B}(C([0,\omega_1]))$ vanishes; we refer to
  Appendix~\ref{section5} for details.
\end{example} 

\begin{example}\label{applicationsofmainthmI}
  Let $p\in (1,\infty)$, and let~$X = J_p$ be the $p^{\text{th}}$
  quasi-reflexive James space, which is defined by $J_p = \{ x\in c_0
  : \|x\|_{J_p}<\infty\}$, where
  \[ \| x\|_{J_p} =
  \sup\biggl\{\Bigl(\sum_{j=1}^m|x_{k_j}-x_{k_{j+1}}|^p\Bigr)^{\frac{1}{p}}
  : m,k_1,\ldots,k_{m+1}\in\N,\,k_1<
  k_2<\cdots<k_{m+1}\biggr\}\in[0,\infty] \] for each scalar sequence
  $x = (x_k)_{k\in\N}$.  (This space was first considered for $p=2$ by
  James~\cite{james} and later generalized to arbitrary~$p$ by
  Edelstein and Mityagin~\cite{em}.) Moreover, let $Y =
  \bigl(\bigoplus_{n\in\N}J_p^{(n)}\bigr)_{\ell_p}$, where $J_p^{(n)}$
  denotes the $n$-dimensional subspace of~$J_p$ consisting of those
  elements which vanish from the $(n+1)^{\text{st}}$ coordinate
  onwards.  Then $X\cong X\oplus Y$ and~$Y$ is isomorphic to the
  $\ell_p$-direct sum of countably many copies of itself by
  \cite[Lemmas~5--6]{em}. (Note, however, that a key condition appears
  to be missing in the statement of \cite[Lemma~5]{em}, namely that
  the sequence denoted by~$\nu$ is unbounded.) Further, the
  ideal~$\mathscr{G}_Y(X)$ has codimension one in~$\mathscr{B}(X)$ by
  \cite[Theorem~4.3]{laustsenCommutators}, so that
  Lemma~\ref{lemma29Nov2013} and Theorem~\ref{thm26Nov2013} show that
  $K_0(\mathscr{B}(J_p))\cong\Z$. This reproves
  \cite[Theorem~4.6]{laustsenKtheory}, whose proof inspired our proof
  of The\-o\-rem~\ref{thm26Nov2013}, above.

  Kochanek and the first-named author have observed that the results
  obtained in \cite[Section~3]{KK} ensure that the proof of
  \cite[Theorem~4.6]{laustsenKtheory} carries over to \emph{Edgar's
    long James space}~$J_p(\omega_1)$, originally introduced
  in~\cite{edgar}, so that $K_0(\mathscr{B}(J_p(\omega_1)))\cong\Z$
  (see \cite[Propo\-si\-tion~3.13]{KK}). Our results provide an
  explicit proof of this conclusion, using~\cite{KK}. Indeed, let $X =
  J_p(\omega_1)$, and define \mbox{$Y = \bigl(\bigoplus_{\alpha\in
      L}J_p(\alpha)\bigr)_{\ell_p}$}, where~$L$ denotes the set of
  countably infinite limit ordinals and~$J_p(\alpha)$ is the closed
  subspace of~$J_p(\omega_1)$ spanned by the indicator functions of
  the ordinal intervals~$[0,\beta]$ for $\beta<\alpha$. Then
  \cite[Propo\-si\-tion~3.3, Lemma~3.4 and Theorem~3.7]{KK} show that
  $X\cong X\oplus Y$, that~$Y$ is isomorphic to the $\ell_p$-direct
  sum of countably many copies of itself, and that the
  ideal~$\mathscr{G}_Y(X)$ has codimension one in~$\mathscr{B}(X)$, so
  that Lemma~\ref{lemma29Nov2013} and Theorem~\ref{thm26Nov2013}
  apply.
\end{example} 

Our final applications of Theorem~\ref{thm26Nov2013} rely on the
following general observation.
\begin{lemma}\label{cor4Dec2013} 
  Suppose that $X = W\oplus Y$, where~$W$ and~$Y$ are Banach spaces
  such that:
  \begin{enumerate}
  \item\label{cor4Dec2013i} $Y$ is isomorphic to the $\ell_p$- or
    $c_0$-direct sum of countably many copies of itself for some
    $p\in[1,\infty);$ and
  \item\label{cor4Dec2013ii} the ideal $\mathscr{G}_Y(W)$ has
    codimension one in~$\mathscr{B}(W)$.
  \end{enumerate}
  Then~$X$ and~$Y$ satisfy
  conditions~\romanref{thm26Nov2013ii}--\romanref{thm26Nov2013i} of
  Theorem~{\normalfont{\ref{thm26Nov2013}}}, and hence
  $K_0(\mathscr{B}(X))\cong\Z$.
\end{lemma}
\begin{proof}
  Condition~\alphref{cor4Dec2013i} ensures that $Y\cong Y\oplus Y$,
  so that $X\cong X\oplus Y$, and Lemma~\ref{lemma29Nov2013} therefore
  shows that
  conditions~\romanref{thm26Nov2013ii}--\romanref{thm26Nov2013iii} of
  Theorem~{\normalfont{\ref{thm26Nov2013}}} are satisfied.

To verify condition~\romanref{thm26Nov2013i}, we use the fact that
each operator $T\in\mathscr{B}(X)$ can be represented as a $(2\times
2)$-matrix
\[ T = \begin{pmatrix} T_{1,1}\colon W\to W & T_{1,2}\colon Y\to W\\
  T_{2,1}\colon W\to Y & T_{2,2}\colon Y\to Y\end{pmatrix}, \] and $T$
factors through~$Y$ if and only if~$T_{j,k}$ does for each pair
$j,k\in\{1,2\}$. Since $T_{j,k}$ trivially factors through~$Y$ for
$(j,k)\ne(1,1)$, condition~\alphref{cor4Dec2013ii} shows
that~$\mathscr{G}_Y(X)$ has codimension one in~$\mathscr{B}(X)$.
\end{proof}

\begin{example}\label{argyroshaydonjohnson}
  Let $W = X_{\text{AH}}$ be Argyros and Haydon's Banach space which
  solves the scalar-plus-compact problem (see~\cite{ah}), and, for
  some $p\in[1,\infty]$, let $Y = C_p$ be Johnson's $p^{\text{th}}$
  universal space with the property that all approximable operators
  factor through~$C_p$ (see~\cite{joh}). Then, as noted
  in~\cite[p.~341]{joh}, $Y$~is (iso\-met\-ri\-cally) isomorphic to
  either the $\ell_p$-direct sum (for $p<\infty$) or the $c_0$-direct
  sum (for $p=\infty$) of countably many copies of itself.

  Moreover, every compact operator~$T$ on~$W$ is approximable
  because~$W$ has a Schauder basis, and therefore~$T$ factors
  through~$Y$ by the fundamental property of~$Y$ (see
  \cite[Theorem~1]{joh}).  Hence we have
  $\mathscr{K}(W)\subseteq\mathscr{G}_Y(W)$.  To show that these two
  ideals are equal, we assume the contrary. Then, as $\mathscr{K}(W)$
  has codimension one in~ $\mathscr{B}(W)$, necessarily
  $I_W\in\mathscr{G}_Y(W)$, so that Lemma~\ref{STSTlemma} implies
  that~$Y$ contains a complemented subspace which is isomorphic
  to~$W$. However, as observed in~\cite[p.~341]{joh}, every closed,
  infinite-dimensional subspace of~$Y$ contains a subspace which is
  iso\-mor\-phic to~$\ell_p$ (for $p<\infty$) or~$c_0$ (for
  $p=\infty$), but no subspace of~$W$ is isomorphic to~$\ell_p$
  or~$c_0$ because~$W$ is hereditarily in\-de\-composable by
  \cite[Theorem~8.11]{ah}.  This contradiction proves that
  $\mathscr{G}_Y(W) = \mathscr{K}(W)$.  In particular
  $\mathscr{G}_Y(W)$ has codimension one in~$\mathscr{B}(W)$, so that
  Lemma~\ref{cor4Dec2013} implies that
  $K_0(\mathscr{B}(X_{\text{AH}}\oplus C_p))\cong\Z$.
\end{example}

\begin{example}\label{ShelahSteprans}
  Let~$W$ be the non-separable Banach space constructed by Shelah
  and Ste\-pr\={a}ns~\cite{ss} such that the ideal~$\mathscr{X}(W)$ of
  operators with separable range has codimension one
  in~$\mathscr{B}(W)$, and choose a family
  $(Y_\gamma)_{\gamma\in\Gamma}$ of closed, separable subspaces of~$W$
  such that:
  \begin{enumerate}
  \item\label{ShelahSteprans1} every closed, separable subspace of~$W$
    is isomorphic to~$Y_\gamma$ for some $\gamma\in\Gamma$; and
  \item\label{ShelahSteprans2} every subspace $Y_\beta$ is repeated
    countably many times in the family $(Y_\gamma)_{\gamma\in\Gamma}$,
    in the sense that the set $\{\gamma\in\Gamma : Y_\gamma =
    Y_\beta\}$ is countably infinite for each $\beta\in\Gamma$.
  \end{enumerate}
  Set $Y = \bigl(\bigoplus_{\gamma\in\Gamma}Y_\gamma\bigr)_{\ell_p}$
  for some $p\in(1,\infty)$. Condition~\alphref{ShelahSteprans2}
  ensures that~$Y$ is isomorphic to the $\ell_p$-direct sum of
  countably many copies of itself.

  We shall now proceed to show that $\mathscr{X}(W) =
  \mathscr{G}_Y(W)$. Indeed, for each $T\in\mathscr{X}(W)$, we can
  choose $\gamma\in\Gamma$ such that there is an isomorphism~$U$
  of~$\overline{T[W]}$ onto~$Y_\gamma$. Let $\iota_\gamma\colon
  Y_\gamma\to Y$ and $\pi_\gamma\colon Y\to Y_\gamma$ denote the
  canonical $\gamma^{\text{th}}$ coordinate embedding and projection,
  respectively. Then we have $T = SR$, where the operators~$R$ and~$S$
  given by $R\colon w\mapsto \iota_\gamma UTw,\, W\to Y,$ and $S\colon
  y\mapsto U^{-1}\pi_\gamma y,\, Y\to W$.  This shows that
  $T\in\mathscr{G}_Y(W)$, and therefore the inclusion
  $\mathscr{X}(W)\subseteq\mathscr{G}_Y(W)$ holds.

  On the other hand, the Banach space~$Y$ is weakly compactly
  generated, so that the same is true for each of its complemented
  subspaces.  Wark \cite[Proposition~2]{wark} has shown that~$W$ is
  not weakly compactly generated. Hence no complemented subspace
  of~$Y$ is isomorphic to~$W$, so that $I_W\notin\mathscr{G}_Y(W)$ by
  Lemma~\ref{STSTlemma}. Thus we conclude that $\mathscr{G}_Y(W) =
  \mathscr{X}(W)$, and therefore Lemma~\ref{cor4Dec2013} shows that
  $K_0(\mathscr{B}(W\oplus Y))\cong\Z$.
\end{example} 

\begin{example}\label{example29Jan2014}
  Assume either the Continuum Hypothesis or Martin's Axiom together
  with the negation of the Continuum Hypothesis, and let $W = C(K)$,
  where~$K$ is the scattered compact Hausdoff space described in
  Example~\ref{mrowkaexample}. Suppose that~$Y$ is a Banach space such
  that:
  \begin{enumerate}
  \item\label{example29Jan2014i} $Y$ is isomorphic to the $\ell_p$- or
    $c_0$-direct sum of countably many copies of itself for some
    $p\in[1,\infty);$ 
  \item\label{example29Jan2014ii} $Y$ contains a complemented subspace
    which is isomorphic to~$c_0$; and
  \item\label{example29Jan2014iii} no complemented subspace of~$Y$ is
    isomorphic to~$W$.
  \end{enumerate}
  Condition~\alphref{example29Jan2014iii} ensures that the
  ideal~$\mathscr{G}_Y(W)$ is proper, while
  condition~\alphref{example29Jan2014ii} implies that it contains the
  ideal~$\mathscr{G}_{c_0}(W)$, which has codimension one
  in~$\mathscr{B}(W)$. Hence~$\mathscr{G}_Y(W) =
  \mathscr{G}_{c_0}(W)$, so that~$\mathscr{G}_Y(W)$ has codimension
  one in~$\mathscr{B}(W)$.  The conditions of Lemma~\ref{cor4Dec2013}
  are therefore satisfied, and thus $K_0(\mathscr{B}(W\oplus
  Y))\cong\Z$.

  For instance, conditions
  \alphref{example29Jan2014i}--\alphref{example29Jan2014ii}, above,
  are satisfied for $Y = C(M)$, where~$M$ is any infinite, compact
  metric space.
\end{example}
\appendix
\section{An alternative approach for
  $\mathscr{B}(C([0,\omega_1]))$}\label{section5}
\noindent
Edelstein and Mityagin stated in \cite[Proposition~4]{em} that the
invertible group of the Banach
algebra~$\mathscr{B}(C([0,\omega_1])^n)$ is homotopy equi\-va\-lent to
the invertible group of scalar-valued $(n\times n)$-matrices for each
$n\in\N$. The aim of this appendix is to show how this result can be
applied to reprove the con\-clu\-sion of
Example~\ref{applicationsofmainthmComega1} that
$K_0(\mathscr{B}(C([0,\omega_1])))\cong\Z$, and also to deduce from it
that the $K_1$-group of~$\mathscr{B}(C([0,\omega_1]))$ vanishes. Note
that this approach works for complex scalars only; we shall therefore
suppose that the scalar field is~$\C$ throughout this appendix.

The $K_1$-group will be the main object of interest, so we shall begin
by defining it formally. In contrast to the purely ring-theoretic
definition of~$K_0$, topology plays a key role here. Suppose
that~$\mathscr{A}$ is a complex, unital Banach algebra.  For each
$n\in\N$, we turn~$M_n(\mathscr{A})$ into a Banach algebra by
identifying it with its natural image in the Banach algebra
$\mathscr{B}(\mathscr{A}^n)$ of operators acting on the direct sum of
$n$ copies of~$\mathscr{A}$, where~$\mathscr{A}^n$ is equipped with
the norm $\bigl\|(A_1,\ldots,A_n)\bigr\| =
\max\bigl\{\|A_1\|,\ldots,\|A_n\|\bigr\}$, as in
Section~\ref{section2}.\smallskip

\noindent
\textsl{Note.} For a Banach space~$E$, we have now equipped
$M_n(\mathscr{B}(E))$ with two potentially different norms, one coming
from its identification with $\mathscr{B}(E^n)$, the other arising
from its embedding into $\mathscr{B}(\mathscr{B}(E)^n)$. Fortunately,
these two norms are equal, as is easily checked. \smallskip

Let $\inv_n(\mathscr{A})$ be the group of invertible elements of
$M_n(\mathscr{A})$, and denote by~$1_{M_n(\mathscr{A})}$ the $(n\times
n)$-identity matrix over~$\mathscr{A}$.  Given
$U\in\inv_m(\mathscr{A})$ and $V\in\inv_n(\mathscr{A})$, where
$m,n\in\N$, we say that~$U$ and~$V$ are \emph{$K_1$-equivalent},
written $U\sim_1 V$, if, for some integer $k \geq \max\{m,n\}$, there
exists a con\-tinuous path $t\mapsto
W_t,\,[0,1]\to\inv_k(\mathscr{A}),$ such that
\begin{equation}\label{K1path} W_0 = \begin{pmatrix} U & 0 \\ 0 &
    1_{M_{k-m}(\mathscr{A})} \end{pmatrix} \qquad\text{and}\qquad W_1
  = \begin{pmatrix} V & 0 \\ 0 &
    1_{M_{k-n}(\mathscr{A})} \end{pmatrix}. \end{equation} This
defines an equivalence relation $\sim_1$ on the set
$\inv_{\infty}(\mathscr{A}) = \bigcup_{n\in\N}\inv_n(\mathscr{A})$,
and the quotient $K_1(\mathscr{A}) =
\inv_{\infty}(\mathscr{A})/\mathord{\sim_1}$ is an abelian group with
respect to the operation
\[ \bigl([U]_1,[V]_1\bigr)\mapsto \left[\begin{pmatrix} U & 0 \\ 0 &
    V \end{pmatrix}\right]_1,\quad K_1(\mathscr{A})\times
K_1(\mathscr{A})\to K_1(\mathscr{A}), \] where $[U]_1$ denotes the
equivalence class of $U\in\inv_{\infty}(\mathscr{A})$ in
$K_1(\mathscr{A})$.

A cornerstone of $K$-theory for complex Banach algebras is that
\emph{Bott periodicity} holds:
\begin{equation}\label{bott} K_0(\mathscr{A})\cong 
  K_1(\widetilde{S\mathscr{A}}), \end{equation} 
where $\widetilde{S\mathscr{A}}$ denotes the \emph{unitization} of the
\emph{suspension} of~$\mathscr{A}$, that is,
\[ \widetilde{S\mathscr{A}} = \{ f\in C([0,1],\mathscr{A}) : f(0) =
f(1)\in\C 1_{\mathscr{A}}\}; \] here $C([0,1],\mathscr{A})$ denotes
the Banach algebra of continuous, $\mathscr{A}$-valued functions
defined on the unit interval~$[0,1]$, and $1_{\mathscr{A}}$ is the
multiplicative identity of~$\mathscr{A}$. We may identify
$M_n(C([0,1],\mathscr{A}))$ with $C([0,1],M_n(\mathscr{A}))$ for each
$n\in\N$; under this identification, we have
\begin{equation}\label{eqinvnSAtilde}
  \inv_n\widetilde{S\mathscr{A}} = 
  \{ f\in C([0,1],\inv_n\mathscr{A}) : f(0) = f(1)\in M_n(\C
  1_{\mathscr{A}})\}. \end{equation}
 
Let~$\mathscr{A}$ and~$\mathscr{C}$ be complex, unital Banach
algebras, and let $\phi\colon\mathscr{A}\to\mathscr{C}$ be a bounded,
unital algebra homomorphism.  We shall require the following two
ho\-mo\-mor\-phisms associated with~$\phi$.  First, in analogy
with~\eqref{K0functor}, we can define a group homomorphism
\mbox{$K_1(\phi)\colon K_1(\mathscr{A})\to K_1(\mathscr{C})$} by
\begin{equation}\label{eq3Jan:defnK1functor} 
  K_1(\phi)([U]_1) = [\phi_n(U)]_1\qquad (n\in\N,\,
  U\in\inv_n\mathscr{A}), \end{equation}
where $\phi_n$ is given by~\eqref{defnphin}, and secondly, 
we obtain a bounded,
unital algebra homomorphism $\widetilde{S\phi}\colon 
\widetilde{S\mathscr{A}}\to\widetilde{S\mathscr{C}}$ 
by the definition $\widetilde{S\phi}(f) = \phi\circ f$.

The following result is probably well known to experts, but since we
have been unable to locate a precise reference to it, we include a
proof.
\begin{lemma}\label{LemmaA1} Let~$\mathscr{A}$ and~$\mathscr{C}$ be  
  complex, unital Banach algebras, let $n\in\N,$ and let
  \mbox{$\phi\colon\mathscr{A}\to\mathscr{C}$} and
  $\psi\colon\mathscr{C}\to\mathscr{A}$ be bounded, unital algebra
  homomorphisms such that the restriction to~$\inv_n\mathscr{A}$ of
  the mapping $\psi_n\circ\phi_n$ is homotopy equivalent to the
  identity mapping, in the sense that there exists a continuous
  mapping $F\colon [0,1]\times\inv_n\mathscr{A}\to\inv_n\mathscr{A}$
  such that $F(0,U) = U$ and $F(1,U) = \psi_n\circ\phi_n(U)$ for each
  $U\in\inv_n\mathscr{A}$. Then
  \[ K_1(\widetilde{S\psi})\circ K_1(\widetilde{S\phi})([f]_1) =
  [f]_1\qquad (f\in\inv_n\widetilde{S\mathscr{A}}). \]
\end{lemma}
\begin{proof} Given $f\in\inv_n\widetilde{S\mathscr{A}}$, we define
  $g_t(r) = F(t,f(0))^{-1}F(t,f(r))\in\inv_n\mathscr{A}$ for each pair
  $r,t\in[0,1]$, where $F$ is chosen as above. An easy check
  using~\eqref{eqinvnSAtilde} shows that
  $g_t\in\inv_n\widetilde{S\mathscr{A}}$ for each $t\in[0,1]$.
  Moreover, the mapping $(r,t)\mapsto g_t(r),\,[0,1]^2
  \to\inv_n\mathscr{A},$ is continuous, and it is therefore uniformly
  continuous, so that, for each $\epsilon>0$, there exists $\delta>0$
  such that $\|g_t(r) - g_{t'}(r')\|_{M_n(\mathscr{A})}\leq\epsilon$
  whenever $r,r',t,t'\in[0,1]$ satisfy
  \mbox{$\max\{|r-r'|,|t-t'|\}\leq\delta$}.  This implies that
  \[ \| g_t-g_{t'}\|_{C([0,1],M_n(\mathscr{A}))} =
  \sup_{r\in[0,1]}\|g_t(r) -
  g_{t'}(r)\|_{M_n(\mathscr{A})}\le\epsilon\qquad
  (t,t'\in[0,1],\,|t-t'|\leq\delta), \] which shows that the mapping
  $t\mapsto g_t,\, [0,1] \to\inv_n\widetilde{S\mathscr{A}},$ is
  continuous. Hence we have
  \begin{equation}\label{Nov2013eqn1} [f(0)^{-1}\cdot f]_1
    =[f(0)^{-1}\cdot(\psi_n\circ\phi_n\circ f)]_1\quad\text{in}\quad
    K_1(\widetilde{S\mathscr{A}})
  \end{equation} because $g_0(r) =
  f(0)^{-1}f(r)$ and \[ g_1(r) =
  (\psi_n\circ\phi_n)(f(0))^{-1}(\psi_n\circ\phi_n)(f(r)) =
  f(0)^{-1}(\psi_n\circ\phi_n\circ f)(r)\qquad (r\in [0,1]), \] 
  where we have used the fact that $\psi_n\circ\phi_n(U) = U$ 
  for each $U\in M_n(\C 1_{\mathscr{A}})$.

  Since $\inv_n(\C 1_{\mathscr{A}})$ is homeomorphic to~$\inv_n\C$, it
  is path-connected. We can therefore choose a continuous mapping
  $t\mapsto V_t,\, [0,1]\to \inv_n(\C 1_{\mathscr{A}}),$ such that
  $V_0 = 1_{M_n(\mathscr{A})}$ and $V_1 = f(0)^{-1}$. This implies
  that the mappings $t\mapsto V_t\cdot f$ and $t\mapsto
  V_t\cdot(\psi_n\circ\phi_n\circ f)$ of~$[0,1]$
  into~$\inv_n\widetilde{S\mathscr{A}}$ are continuous. They connect
  $f$ with $f(0)^{-1}\cdot f$ and $\psi_n\circ\phi_n\circ f$ with
  $f(0)^{-1}\cdot(\psi_n\circ\phi_n\circ f)$, respectively.  When
  combined with~\eqref{Nov2013eqn1}, this shows that
  \[ [f]_1 = [f(0)^{-1}\cdot f]_1
  =[f(0)^{-1}\cdot(\psi_n\circ\phi_n\circ f)]_1 =
  [\psi_n\circ\phi_n\circ f]_1 = K_1(\widetilde{S\psi})\circ
  K_1(\widetilde{S\phi})([f]_1), \] as required.
\end{proof}

\begin{corollary}\label{corK0andK1}
  Let~$\mathscr{A}$ and~$\mathscr{C}$ be complex, unital Banach
  algebras, and suppose that there exist bounded, unital algebra
  homomorphisms $\phi\colon\mathscr{A}\to\mathscr{C}$ and
  $\psi\colon\mathscr{C}\to\mathscr{A}$ such that the restrictions
  $\phi_n\colon\inv_n\mathscr{A}\to\inv_n\mathscr{C}$ and
  $\psi_n\colon\inv_n\mathscr{C}\to\inv_n\mathscr{A}$ induce a
  homotopy equivalence for each $n\in\N$, in the sense that
  $\psi_n\circ\phi_n$ is homotopy equivalent to the identity mapping
  on~$\inv_n\mathscr{A}$ and $\phi_n\circ\psi_n$ is homotopy
  equivalent to the identity mapping on~$\inv_n\mathscr{C}$. Then
  \begin{equation}\label{eqK0andK1}
    K_0(\mathscr{A})\cong K_0(\mathscr{C})\qquad \text{and}\qquad 
    K_1(\mathscr{A})\cong K_1(\mathscr{C}). \end{equation}
\end{corollary}

\begin{proof}
  Using the assumptions in tandem with Lemma~\ref{LemmaA1}, we see
  that $K_1(\widetilde{S\phi})$ is an isomorphism
  of~$K_1(\widetilde{S\mathscr{A}})$
  onto~$K_1(\widetilde{S\mathscr{C}})$ with
  inverse~$K_1(\widetilde{S\psi})$. Hence we have
  \[ K_0(\mathscr{A})\cong K_1(\widetilde{S\mathscr{A}})\cong
  K_1(\widetilde{S\mathscr{C}})\cong K_0(\mathscr{C}) \] by two
  applications of Bott periodicity~\eqref{bott}. This establishes the
  first part of~\eqref{eqK0andK1}.
 
  The second part is much simpler. Indeed, working straight from the
  definitions~\eqref{K1path} and~\eqref{eq3Jan:defnK1functor}, we
  obtain
  \[ K_1(\psi)\circ K_1(\phi)([U]_1) = [\psi_n\circ\phi_n(U)]_1 =
  [U]_1\qquad (n\in\N,\, U\in\inv_n\mathscr{A}) \] because
  $\psi_n\circ\phi_n$ is homotopy equivalent to the identity mapping
  on~$\inv_n\mathscr{A}$.  A similar argument shows that
  $K_1(\phi)\circ K_1(\psi)$ is equal to the identity
  on~$K_1(\mathscr{C})$, and $K_1(\phi)$ is there\-fore an isomorphism
  of $K_1(\mathscr{A})$ onto~$K_1(\mathscr{C})$ with
  inverse~$K_1(\psi)$.
\end{proof}

Recall from Example~\ref{applicationsofmainthmComega1} that, for $Y =
\bigl(\bigoplus_{\alpha<\omega_1} C([0,\alpha])\bigr)_{c_0}$, the
ideal $\mathscr{G}_Y(C([0,\omega_1]))$ has codimension one in
$\mathscr{B}(C([0,\omega_1]))$, and let
$\phi\colon\mathscr{B}(C([0,\omega_1]))\to\C$ be the corresponding
algebra homomorphism given by~\eqref{eqDefnphi}.  Further, define
\[ \psi\colon\ \lambda\mapsto \lambda I_{C([0,\,\omega_1])},\quad
\C\to\mathscr{B}(C([0,\omega_1])). \] Then $\phi$ and $\psi$ are
bounded, unital algebra homomorphisms such that $\phi\circ\psi =
I_{\C}$, and hence $\phi_n\circ\psi_n = I_{M_n(\C)}$ for each
$n\in\N$. The precise statement of \cite[Proposition~4]{em} is that
$\psi_n\circ\phi_n$ is homotopy equivalent to the identity mapping
on~$\inv_n\mathscr{B}(C([0,\omega_1]))$ for each $n\in\N$, so that we
may apply Corollary~\ref{corK0andK1} to obtain
\[ K_0(\mathscr{B}(C([0,\omega_1])))\cong K_0(\C)\cong\Z\qquad
\text{and}\qquad K_1(\mathscr{B}(C([0,\omega_1])))\cong K_1(\C) =
\{0\}. \] This completes the proof of the statements made at the
beginning of this appendix.

A comparison between the above calculation of the $K_0$-group
of~$\mathscr{B}(C([0,\omega_1]))$ and the one given in
Example~\ref{applicationsofmainthmComega1} shows some obvious
advantages of the former, namely that it is shorter and simultaneously
leads to the determination of the $K_1$-group; however, it also has
some significant drawbacks:
\begin{itemize}
\item it does not apply to real scalars;
\item it relies on some very heavy machinery, notably Bott
  periodicity, but also Edelstein and Mityagin's highly non-trivial
  result;
\item it is entirely topological, despite the purely ring-theoretic
  nature of~$K_0$.
\end{itemize}

\section*{Acknowledgements} 
\noindent
Part of this work was carried out during a visit of the second author
to Lancaster in February 2012, supported by a London Mathematical
Society Scheme 2 grant (ref.~21101). The authors gratefully
acknowledge this support. The second author was also partially
supported by the Polish National Science Center research grant
2011/01/B/ST1/00657.

\bibliographystyle{amsplain}

\bigskip
\end{document}